\documentclass[10pt,reqno]{amsart}
\usepackage[T2A]{fontenc}
\usepackage{amsmath}
\usepackage{amssymb}
\usepackage{amsfonts}
\usepackage{bm} 
\usepackage{stmaryrd} 
\usepackage{calrsfs}  
\usepackage[pic]{xy}  

\allowdisplaybreaks 

\usepackage{xcolor} 

\newtheorem{thm}{Theorem}
\newtheorem{prop}[thm]{Proposition}
\newtheorem{lem}[thm]{Lemma}

\newtheorem{prob}{Problem}

\begin{document}

\title[Explicit inverse Shapiro isomorphism]
{Explicit inverse Shapiro isomorphism \\ and its application}

\author{Andrei V. Zavarnitsine}%
\address{Andrei V. Zavarnitsine
\newline\indent Sobolev Institute of Mathematics,
\newline\indent 4, Koptyug av.
\newline\indent 630090, Novosibirsk, Russia
} \email{zav@math.nsc.ru}

\maketitle

\begin{quote}
\noindent{{\sc Abstract.}
We find an explicit form of the inverse isomorphism from Shapiro's lemma in terms of
inhomogeneous cocycles and apply it to construct special nonsplit coverings of
groups with a unique conjugacy class of involutions.}


\end{quote}

\bigskip

\hfill{\sl To Prof. Victor D. Mazurov on his 80th birthday}

\section{Introduction}

Let $G$ be a group, $H\leqslant G$ a subgroup, and let $U$ be a $\mathbb{Z}H$-module.
The co-induced module $W=\operatorname{Coind}_H^G(U)$ is by definition
$\operatorname{Hom}_{\mathbb{Z}H}(\mathbb{Z}G,U)$ viewed as a
right $G$-module.
Shapiro's well-known and useful lemma from homological algebra \cite[Lemma 6.3.2]{Wei94}
states that there is an isomorphism of cohomology groups
$$
  \Theta:\ H^n(G,W)\to H^n(H,U)
$$
for all $n\geqslant 0$. If $H^n$ is expressed as the quotient $Z^n/B^n$ for {\em inhomogeneous}
cocycle and coboundary groups $Z^n$ and $B^n$ then it can be seen that
$\Theta$ is induced from the map on cocycles
$$
\alpha\ \mapsto\ [(h_1,\ldots,h_n)\mapsto\alpha(h_1,\ldots,h_n)(1)]\in Z^n(H,U)
$$
for every $\alpha\in Z^n(G,W)$ and $h_1,\ldots,h_n\in H$. A corresponding map
for the inverse isomorphism $\Theta^{-1}$ is not so obvious.

Let $Y\subseteq G$ be a {\em left} transversal of $H$ in $G$. And let $\pi:G\to H$ be
the projection along~$Y$, i.\,e. $\pi(yh)=h$ for all $y\in Y$, $h\in H$. We prove

\begin{thm}\label{main}
  In the above notation, the inverse Shapiro isomorphism $\Theta^{-1}$
is induced from the map on inhomogeneous cocycles
$$
\beta\ \mapsto\ [(g_1,\ldots,g_n)\mapsto[g_{n+1}\mapsto \beta(h_1,\ldots,h_n)h_{n+1} ]] \in Z^n(G,W)
$$
for every $\beta\in Z^n(H,U)$ and $g_1,\ldots,g_{n+1}\in G$, where the elements
$h_i\in H$
are uniquely determined from the relations
$$
h_ih_{i+1}\ldots h_{n+1} = \pi (g_ig_{i+1}\ldots g_{n+1})
$$
for $i=1,\ldots,n+1$.
\end{thm}

We remark that in some form the explicit inverse Shapiro isomorphism has been determined. For example, in the case of a finite group $G$
the expression in terms of {\em homogeneous} cochains can be found in \cite[Theorem 12.14]{Gui18}, where left modules and left action are used.
We do not see a ready way to translate this to our setting.
Our choice of right group action and modules is more in line with many classic group-theoretic books, e.\,g. the Atlas \cite{Atl85},
as well as conventions in computer software such as {\sf GAP} and {\sf Magma}, whereas we use inhomogeneous cocycles because in
low dimensions they yield constructions of group extensions. It is also not necessary to require that $G$ be finite.

As an application, we use a special case of Theorem 1 for $n=2$ to prove the
following result related to a problem by V.\,D. Mazurov about Suzuki groups, see Section \ref{sec:appl} for details.

\begin{prop}\label{mprop}
Let $G$ be a finite group with a unique conjugacy class of involutions and let $H\leqslant G$
be a subgroup of order $2$.  There is an extension $E$ of a $2$-group $T$ such that all involutions
of $E$ lie in $T$.
\end{prop}

The extension $E$ from Proposition \ref{mprop} is constructed explicitly. We show that
one may take for $T$ an induced $\mathbb{F}_2H$-module.

\section{Preliminaries}

For sets $A,B,C$, the (right) composition maps $\alpha:A\to B$ and $\beta:B\to C$ is denoted by $\alpha\beta$.
The image of $a\in A$ under $\alpha$ is denoted by either $\alpha(a)$ or $a\alpha$ depending on the context.
Maps are sometimes written in the form $[a\mapsto a\alpha]$. We use right modules and right group actions.

First, we recall some necessary notions and facts from homological algebra. For more details,
see \cite{Bro82,Ver03,Wei94}.
Given a ring $R$, a sequence of right $R$-modules and homomorphisms
\begin{equation}\label{chc}
\qquad \ldots\stackrel{\delta_{n+1}}{\longrightarrow} A_n
      \stackrel{\delta_n}{\longrightarrow}     A_{n-1}
      \stackrel{\delta_{n-1}}{\longrightarrow}\ldots,
\end{equation}
written $(A,\delta)$ or $A$ for short, is a {\em chain complex} if $\delta_n\delta_{n-1}=0$ for all $n$.
{\em The homology} $H(A)$ of $A$ is the collection of $R$-modules $H_n(A)=Z_n(A)/B_n(A)$,
the {\em  $n$-dimensional homology groups}, where $Z_n=\operatorname{Ker}\delta_n$
and $B_n=\operatorname{Im}\delta_{n-1}$.
The complex $A$ is {\em exact} (at term $n$) if $H(A)=0$ ($H_n(A)=0$, respectively).
A {\em chain map} $f:(A,\delta)\to (B,\gamma)$ of two chain complexes is a set of
$R$-homomorphisms $f_n:A_n\to B_n$ such that $f_n\gamma_n=\delta_n f_{n+1}$ for all $n$.
Such an $f$ induces a map $H(f):H(A)\to H(B)$, i.\,e. $R$-homomorphisms of $n$-dimensional homology groups for all $n$.
We denote by $\operatorname{id}_A$ the obvious identity chain map $A\to A$.
Two chain maps $f,g:(A,\delta)\to (B,\gamma)$ are {\em homotopic} if there is a collection $\eta$
of $R$-homomorphisms $\eta_n: A_n\to B_{n+1}$ such that $f_n-g_n=\eta_n\gamma_{n+1}+ \delta_n \eta_{n-1}$
for all~$n$. In this case, $\eta$ is called a {\em chain homotopy} from $f$ to $g$.
A chain map $f:A\to B$ is a {\em homotopy equivalence} if there is a chain map $f':B\to A$ such that
$ff'$ is chain homotopic to $\operatorname{id}_A$ and $f'f$ is chain homotopic to $\operatorname{id}_B$.

\begin{lem}\cite[Proposition(0.2)]{Bro82}\label{heq}
A homotopy equivalence of complexes induces an isomorphism of their homologies.
\end{lem}

{\em Cochain complexes, maps, cohomology}, etc. are defined similarly by reversing the arrows in (\ref{chc}) and usually
changing subscripts to superscripts.

When $R=\mathbb{Z}G$ for a group $G$, instead of saying $R$-modules and $R$-ho\-mo\-mor\-phisms,
we will often say $G$-modules, $G$-homomorphisms, etc.

For commutative rings $R$ and $S$, the categories of right (left) $R$-modules and $(R,S)$-bimodules are denoted by
$\operatorname{Mod}$-$R$ ($R$-$\operatorname{Mod}$) and $R$-$\operatorname{Mod}$-$S$, respectively.
The following fundamental result is known as the tensor-hom adjunction.
The explicit form of the isomorphism for left modules is given, e.\,g., in  \cite[Lemma 2.8.2]{Ben98}.
We state it here for right modules.

\begin{lem}\label{thadj}
Given rings $R$ and $S$, let $M\in \operatorname{Mod}$-$R$,
$N\in R$-$\operatorname{Mod}$-$S$, and $K\in \operatorname{Mod}$-$S$. Then there is a canonical isomorphism of abelian groups
$$\operatorname{Hom}(M\otimes N,K) \cong \operatorname{Hom}(M,\operatorname{Hom}(N,K))$$
explicitly given by the mutually inverse maps
\begin{equation}\label{thexpl}
\begin{aligned}
  \lambda&\ \mapsto\ [m\mapsto[n\mapsto\lambda(m\otimes n)]], \\
  \mu&\ \mapsto\  [m\otimes n\mapsto \mu(m)(n)]
\end{aligned}
\end{equation}
for all $\lambda\in \operatorname{Hom}(M\otimes N,K)$, $\mu\in \operatorname{Hom}(M,\operatorname{Hom}(N,K))$, $m\in M$, and $n\in N$.
\end{lem}

\section{Inhomogeneous resolution}

Let $G$ be a group and let $X_n$, $n\geqslant 0$, be a right free $G$-module with free generators $(g_1,\ldots,g_n)$, $g_i\in G$, $g_i\ne 1$.
In particular, every element of $X_n$ is a $\mathbb{Z}$-linear combination of elements of the form $(g_1,\ldots,g_n)g_{n+1}$ with $g_i\in G$.
Observe that $X_0\cong \mathbb{Z}G$ is generated by the empty tuple $()$. For convenience, we also assume $(g_1,\ldots,g_n)=0$ if $g_i=1$ for some $i$.
The {\em boundary homomorphism} $d_n:X_n\to X_{n-1}$, $n\geqslant 1$, is given on the generators by
\begin{equation}
\begin{aligned}
\label{d_act}
(g_1,\ldots,g_n)d_n &= (-1)^n(g_2,\ldots,g_n)\\
+\sum_{i=1}^{n-1}(-1)^{n-i}&(g_1,\ldots,g_{i-1},g_ig_{i+1},g_{i+2},\ldots,g_n)+(g_1,\ldots,g_{n-1})g_n
\end{aligned}
\end{equation}
and extended to $X_n$ by $G$-linearity. Similarly, $\varepsilon:X_0\to \mathbb{Z}$ is given by $\varepsilon:()\mapsto 1$. It is verified that
\begin{equation}\label{fr_res}
\ldots \to X_2\stackrel{d_2}{\longrightarrow}{X_1}\stackrel{d_1}{\longrightarrow}{X_0}\stackrel{\varepsilon}{\longrightarrow}\mathbb{Z}\to 0
\end{equation}
is an exact chain complex which is called {\em the free  inhomogeneous (standard, normalised) resolution} for the trivial $G$-module $\mathbb{Z}$ (or {\em resolution for $G$}, for short).

Given a $G$-module $A$, consider the cochain complex of abelian groups
\begin{equation}\label{hcompl}
0\stackrel{\partial_{-1}}{\longrightarrow}\operatorname{Hom}_G(X_0,A)\stackrel{\partial_{0}}{\longrightarrow}\operatorname{Hom}_G(X_1,A)
 \stackrel{\partial_{1}}{\longrightarrow}\ldots
\end{equation}
obtained by applying the left exact contravariant functor $\operatorname{Hom}_G(-,A)$ to the truncated version of~(\ref{fr_res}).
Setting $Z^n(G,A)=\operatorname{Ker}\partial_n$
and $B^n(G,A)=\operatorname{Im}\partial_{n-1}$ we define the $n$th {\em cohomology group},
$n\geqslant 0$, by
$$
H^n(G,A)=Z^n(G,A)/B^n(G,A).
$$

Denote by $C^n(G,A)$ the set of all {\em (inhomogeneous, normalised) $n$-cochains}, i.\,e. all the
maps $\lambda: G^n\to A$, $n\geqslant 0$, such that $\lambda(g_1,\ldots,g_n)=0$ whenever $g_i=1$ for some $i$.
Every such $\lambda$ uniquely defines an element of $\operatorname{Hom}_G(X_n,A)$ and the converse is also true.
In particular, there are isomorphisms of abelian groups
$C^n(G,A)\cong \operatorname{Hom}_G(X_n,A)$
which allow us to rewrite (\ref{hcompl}) as
$$
0\longrightarrow C^0(G,A)\stackrel{\partial_{0}}{\longrightarrow}C^1(G,A)
 \stackrel{\partial_{1}}{\longrightarrow}\ldots,
$$
where $\partial_{n-1}:C^{n-1}(G,A)\to C^n(G,A)$, $n\geqslant 1$, acts by
\begin{multline*}
(\lambda\partial_{n-1})(g_1,\ldots,g_n) = (-1)^n\lambda(g_2,\ldots,g_n)\\
+\sum_{i=1}^n(-1)^{n-i}\lambda(g_1,\ldots,g_{i-1},g_ig_{i+1},g_{i+2},\ldots,g_n)
  +\lambda(g_1,\ldots,g_{n-1})g_n
\end{multline*}
for $\lambda\in C^{n-1}(G,A)$. Thus we may identify $Z^n(G,A)$ and $B^n(G,A)$ with their isomorphic
subgroups of $C^n(C,A)$ which consist of {\em $n$-cocycles} and {\em $n$-coboundaries}, respectively.

It is known \cite[\S 2.4]{Ben98} that if we replace the above standard free resolution $X$
with any projective resolution of $\mathbb{Z}$ then the similarly defined cohomology groups
will be isomorphic to $H^n(G,A)$.

\section{Homotopy of resolutions for a subgroup}

Let $G$ be a group and $H\leqslant G$ a subgroup.
Let $(F,d)$ be the standard free resolution
for $G$ restricted to $H$. Then it is also a free resolution for $H$, nonstandard, in general.
Also, let $(S,d')$ be the standard free resolution for $H$.
We will require an explicit form of the isomorphism of cohomologies
for $H$ obtained from $F$ and $S$.
In order to find it, we determine explicitly a homotopy equivalence of these resolutions, cf. Lemma \ref{heq}.
Namely, we define two maps $f:F\to S$ and
$f':S\to F$ that fit in the following diagram:
\begin{equation}\label{FfS}
\begin{array}{c@{}c@{}c@{}c@{}ccc}
\ldots\longrightarrow&F_1&\stackrel{d_1}{\longrightarrow}&F_0&\stackrel{\varepsilon}{\longrightarrow}&\mathbb{Z}&\longrightarrow 0 \\[5pt]
                     &f_1\Big\downarrow\Big\uparrow f_1'&&f_0\Big\downarrow\Big\uparrow f_0'&&\|& \\
\ldots\longrightarrow&S_1&\stackrel{d_1'}{\longrightarrow}&S_0&\stackrel{\varepsilon'}{\longrightarrow}&\mathbb{Z}&\longrightarrow 0
\end{array}
\end{equation}
Fixing a left transversal $Y$ of $H$ in $G$ we let $\pi:G\to H$ be the projection along $Y$, i.\,e.
such that $\pi(yh)=h$ for all $y\in Y$, $h\in H$. In particular,
\begin{equation}\label{pigh}
  \pi(gh)=\pi(g)h, \quad \text{for all} \ \ g\in G,\ h\in H.
\end{equation}
Let $f'$ be the natural identical embedding $S\to F$ and let $f$ be defined by
\begin{equation}\label{fn}
f_n: (g_1,\ldots,g_n)g_{n+1} \mapsto (h_1,\ldots,h_n)h_{n+1}
\end{equation}
for $n\geqslant 0$ and extended by $\mathbb{Z}$-linearity to the whole of $F_n$, where $h_i\in H$ are uniquely determined from the relations
\begin{equation}\label{hrel}
h_ih_{i+1}\ldots h_{n+1}=\pi(g_ig_{i+1}\ldots g_{n+1}), \qquad i=1,\ldots,n+1.
\end{equation}
In other words, $h_i=\pi(g_i\ldots g_{n+1})\pi(g_{i+1}\ldots g_{n+1})^{-1}$, for $i=1,\ldots,n$, and $h_{n+1}=\pi(g_{n+1})$.

\begin{lem} \label{ffp_ch_m}
$f$ and $f'$ are chain maps.
\end{lem}
\begin{proof}
We only need to check that the components of $f$ and $f'$ are $H$-ho\-mo\-mor\-phisms. This is clear for $f'$.
Conversely, $F_n$ is freely generated as an $H$-module by the elements $(g_1,\ldots,g_n)y$ for $g_i\in G$,
$y\in Y$. Hence, the map
$$
(g_1,\ldots,g_n)y \mapsto (h_1,\ldots,h_n),
$$
where $h_i$ are determined from (\ref{hrel}) with $g_{n+1}=y$, extends to an $H$-ho\-mo\-mor\-phism~$f_n$.
In particular, for every $h\in H$, we have
$$
f_n: (g_1,\ldots,g_n)yh \mapsto (h_1,\ldots,h_n)h.
$$
This is the required map (\ref{fn}), because (\ref{hrel}) holds with $g_{n+1}=yh$ and $h_{n+1}=h$ in view of (\ref{pigh}).
\end{proof}

Observe that in general $f$ and $f'$ are not mutually inverse. Only if $Y\cap H=\{1\}$
do we have $f'f=\operatorname{id}_S$. But even in this case,  $ff'\ne \operatorname{id}_F$, in general.
However, the next result shows that $f$ and $f'$ are almost as good as mutually inverse.

\begin{prop} $f'f$ is homotopic to $\operatorname{id}_S$ and $ff'$ is homotopic to $\operatorname{id}_F$.
\end{prop}
\begin{proof} First, we consider the chain map $\psi' = f'f - \operatorname{id}_S$.

By (\ref{pigh}), we have $\pi(h)=y_0^{-1}h$ for $h\in H$, where $\{y_0\}=Y\cap H$.
In particular, if in (\ref{hrel})
all $g_i\in H$, we have $h_i=g_i^{y_0}$, $i=1,\ldots, n$,  and $h_{n+1}=y_0^{-1}g_{n+1}$.
Therefore, $\psi'$ acts on the free $H$-generators of $S$ by
\begin{equation}\label{psip_act}
\psi'_n: (h_1,\ldots,h_n)\mapsto (h_1^{y_0}, \ldots, h_n^{y_0})y_0^{-1}-(h_1,\ldots,h_n).
\end{equation}
We need to find a chain homotopy $\eta'$ from $f'f$ to $\operatorname{id}_S$, i.\,e.
$H$-ho\-mo\-mor\-phisms $\eta'_n:S_n\to S_{n+1}$ that satisfy
\begin{equation}\label{etap}
\psi_n'=\eta_n'd_{n+1}'+d_n'\eta_{n-1}'
\end{equation}
for $n\geqslant 0$, where  $\eta_{-1}':\mathbb{Z}\to S_0$ is the zero map and $d_0'=\varepsilon'$, as can be seen from the diagram
$$
\xymatrix{
&\ar@{}[l]|{\textstyle{\cdots}} \ar[r]^{d_2'}S_2\ar[d]_{\psi'_2} &S_1\ar[r]^{d_1'}\ar[d]_{\psi'_1}\ar@{.>}[dl]_{\eta_1'}
&S_0\ar[r]^{\varepsilon'}\ar[d]_{\psi'_0}\ar@{.>}[dl]_{\eta_0'}&\mathbb{Z}\ar[r]\ar@{.>}[dl]_{0}\ar[d]_0&  0 \\
&\ar@{}[l]|{\textstyle{\cdots}} \ar[r]_{d_2'}S_2       &S_1\ar[r]_{d_1'} &S_0\ar[r]_{\varepsilon'}&\mathbb{Z}\ar[r]& 0
}$$
We define
\begin{equation}\label{etap_act}
\eta_n': (h_1,\ldots,h_n)=\sum_{i=0}^{n}(-1)^{n-i}(h_1^{y_0},\ldots,h_i^{y_0},y_0^{-1},h_{i+1},\ldots,h_n)
\end{equation}
for all $h_i\in H$ and $n\geqslant 0$, and extend this action to $S$ by $H$-linearity.
In particular, $\eta_0'=[()\mapsto(y_0^{-1})]$. From (\ref{d_act}) and (\ref{etap_act}), we have
\begin{multline}\label{eta_d}
(h_1,\ldots,h_n)\eta_n'd_{n+1}'=
\sum_{i=0}^{n}(-1)^{n-i}(h_1^{y_0},\ldots,h_i^{y_0},y_0^{-1},h_{i+1},\ldots,h_n)d_{n+1}'\\
=(-1)^n\big[(-1)^{n+1}(h_1,\ldots,h_n)\stackrel{\text{\color{blue}[5]}}{+}(-1)^n(y_0^{-1}h_1,h_2,\ldots,h_n)\\
\stackrel{\text{\color{blue}[3]}}{+}\sum_{j=1}^{n-1}(-1)^{n-j}(y_0^{-1},h_1,\ldots,h_jh_{j+1},\ldots,h_n)
\stackrel{\text{\color{blue}[2]}}{+}(y_0^{-1},h_1,\ldots,h_{n-1})h_n\big]\\
+\sum_{i=1}^{n-1}(-1)^{n-i}\big[\stackrel{\text{\color{blue}[1]}}{\phantom{+}}(-1)^{n+1}(h_2^{y_0},\ldots,h_i^{y_0},y_0^{-1},h_{i+1},\ldots,h_n)\\
\stackrel{\text{\color{blue}[4]}}{+}\sum_{j=1}^{i-1} (-1)^{n+1-j}(h_1^{y_0},\ldots,h_j^{y_0}h_{j+1}^{y_0},\ldots,h_i^{y_0},y_0^{-1},h_{i+1},\ldots,h_n)\\
\stackrel{\text{\color{blue}[5]}}{+}(-1)^{n+1-i}(h_1^{y_0},\ldots,h_i^{y_0}y_0^{-1},h_{i+1},\ldots,h_n)\\
\stackrel{\text{\color{blue}[5]}}{+}(-1)^{n-i}(h_1^{y_0},\ldots,h_i^{y_0},y_0^{-1}h_{i+1},\ldots,h_n)\\
\stackrel{\text{\color{blue}[3]}}{+}\sum_{j=i+1}^{n-1}(-1)^{n-j}(h_1^{y_0},\ldots,h_i^{y_0},y_0^{-1},h_{i+1},\ldots,h_jh_{j+1},\ldots,h_n)\\
\stackrel{\text{\color{blue}[2]}}{+}(h_1^{y_0},\ldots,h_i^{y_0},y_0^{-1},h_{i+1},\ldots,h_{n-1})h_n\big]
+\big[(-1)^{n+1}(h_2^{y_0},\ldots,h_n^{y_0},y_0^{-1})\\
\stackrel{\text{\color{blue}[4]}}{+}\sum_{j=1}^{n-1}(-1)^{n+1-j}(h_1^{y_0},\ldots,h_j^{y_0}h_{j+1}^{y_0},\ldots,h_n^{y_0},y_0^{-1})\\
\stackrel{\text{\color{blue}[5]}}{-}(h_1^{y_0},\ldots,h_{n-1}^{y_0},h_n^{y_0}y_0^{-1})+(h_1^{y_0},\ldots,h_n^{y_0})y_0^{-1})\big].
\end{multline}

We also have
\begin{multline}\label{d_eta}
(h_1,\ldots,h_n)d_n'\eta_{n-1}'=\\
\big[ (-1)^n(h_2,\ldots,h_n)
+\sum_{i=1}^{n-1}(-1)^{n-i}(h_1,\ldots,h_{i-1},h_ih_{i+1},h_{i+2},\ldots,h_n) \\
  +(h_1,\ldots,h_{n-1})h_n \big ] \eta_{n-1}' = \big[  (-1)^n(h_2,\ldots,h_n)\\
+\sum_{i=1}^{n-1}(-1)^{n-i}(h_1,\ldots,h_ih_{i+1},\ldots,h_n)
+(h_1,\ldots,h_{n-1})h_n \big] \eta_{n-1}' \\
=\stackrel{\text{\color{blue}[1]}}{\phantom{+}}(-1)^n\sum_{i=1}^{n}(-1)^{n-i}(h_2^{y_0},\ldots,h_i^{y_0},y_0^{-1},h_{i+1},\ldots,h_n)\\
\sum_{i=1}^{n-1}(-1)^{n-i}\big[\stackrel{\text{{\color{blue}\scriptsize [3]}}}{\phantom{+}}
\sum_{j=0}^{i-1}(-1)^{n-1-j}(h_1^{y_0},\ldots,h_j^{y_0},y_0^{-1},h_{j+1},\ldots,h_ih_{i+1},\ldots,h_n)\\
\stackrel{\text{{\color{blue}\scriptsize [4]}}}{+}
\sum_{j=i+1}^{n}(-1)^{n-j}(h_1^{y_0},\ldots,h_i^{y_0}h_{i+1}^{y_0},\ldots,h_j^{y_0},y_0^{-1},h_{j+1},\ldots,h_n)\big]\\
\stackrel{\text{{\color{blue}\scriptsize [2]}}}{+}
\sum_{i=0}^{n-1}(-1)^{n-1-i}(h_1^{y_0},\ldots,h_i^{y_0},y_0^{-1},h_{i+1},\ldots,h_{n-1})h_n.
\end{multline}
Expanding the brackets and adding up the right-hand sides of (\ref{eta_d}) and (\ref{d_eta}), we see that the summands
with the same blue marks cancel out. The remaining terms give
$$
(h_1,\ldots,h_n)(\eta_n'd_{n+1}'+d_n'\eta_{n-1}')= -(h_1,\ldots,h_n)+(h_1^{y_0},\ldots,h_n^{y_0})y_0^{-1}\!=(h_1,\ldots,h_n)\psi'_n
$$
by (\ref{psip_act}). Hence, $\eta'$ is the required chain homotopy.

We how consider the chain map $\psi = ff'-\operatorname{id}_F$ which acts on the free $\mathbb{Z}$-generators of $F$ by
\begin{equation}\label{psi_act}
\psi_n: (g_1,\ldots,g_n)g_{n+1}\mapsto (h_1,\ldots, h_n)h_{n+1}-(g_1,\ldots,g_n)g_{n+1},
\end{equation}
$n\geqslant 0$, where the $h_i$'s are determined from (\ref{hrel}). The fact that this is indeed a chain map follows from Lemma \ref{ffp_ch_m}.

Again, we need to find a homotopy $\eta$ from $ff'$ to $\operatorname{id}_F$,
i.\,e. $H$-homomorphisms $\eta_n:F_n\to F_{n+1}$ such that
\begin{equation}\label{eta}
\psi_n=\eta_n d_{n+1}+d_n\eta_{n-1}
\end{equation}
for $n\geqslant 0$, where  $\eta_{-1}:\mathbb{Z}\to F_0$ is the zero map and $d_0=\varepsilon$, as shown in the diagram
$$
\xymatrix{
&\ar@{}[l]|{\textstyle{\cdots}} \ar[r]^{d_2}F_2 &F_1\ar[r]^{d_1}
&F_0\ar[r]^{\varepsilon}&\mathbb{Z}\ar[r]&  0 \\
&\ar@{}[l]|{\textstyle{\cdots}} \ar[r]_{d_2}F_2\ar[u]_{\psi_2} &F_1\ar[r]_{d_1}\ar@{.>}[ul]_{\eta_1}\ar[u]_{\psi_1} &
F_0\ar[r]_{\varepsilon}\ar@{.>}[ul]_{\eta_0}\ar[u]_{\psi_0}&\mathbb{Z}\ar[r]\ar@{.>}[ul]_{0}\ar[u]_0& 0
}
$$
We first define a map $\rho: G\to Y^{-1}$ by $\rho(g)=\pi(g)g^{-1}$ for all $g\in G$. Observe that
\begin{equation}\label{rhogh}
  \rho(gh)=\rho(g), \quad \text{for all} \ \ g\in G,\ h\in H.
\end{equation}
Now, we set
\begin{equation}\label{eta_act}
\eta_n: (g_1,\ldots,g_n)g_{n+1} \mapsto \sum_{i=1}^{n+1}(-1)^{n+1-i}(h_1,\ldots,h_{i-1},r_i,g_i,\ldots,g_n)g_{n+1},
\end{equation}
for $n\geqslant 0$ and all $g_i\in G$, where the $h_i$'s are given by (\ref{hrel}) and
\begin{equation}\label{rrel}
r_i=\rho(g_ig_{i+1}\ldots g_{n+1}), \qquad i=1,\ldots,n+1.
\end{equation}
Definition (\ref{eta_act}) is consistent. Indeed, the $H$-linearity
follows from (\ref{rhogh}) and (\ref{hrel}).
Moreover, if $g_i=1$ for some $i=1,\ldots,n$ then $h_i=1$ by (\ref{hrel}) and
so the right-hand side of (\ref{eta_act}) vanishes.
Therefore, we have $H$-homomorphisms $\eta_n:F_n\to F_{n+1}$, $n\geqslant 0$.
Also note that $\eta$ is not $G$-linear in general despite what may seem from~(\ref{eta_act}). For example, we have
$\eta_0=[\,()g\mapsto(\rho(g))g\,]$, $g\in G$, and if $g\not\in H$ then $\rho(g)\ne \rho(1)$ which yields
$()g\eta_0\ne ()\eta_0g$.

Observe that (\ref{rrel}) and (\ref{hrel}) imply the relations
\begin{equation}\label{hrrg}
  h_ir_{i+1}=r_ig_i,\ \  i=1,\ldots,n; \quad h_{n+1}=r_{n+1}g_{n+1}.
\end{equation}

For arbitrary $g_1,\ldots,g_{n+1}\in G$, using the $G$-linearity of $d_{n+1}$ we have
\begin{multline}\label{getad}
  (g_1,\ldots,g_n)g_{n+1}\eta_n d_{n+1} = \big[ \sum_{i=1}^{n+1}(-1)^{n+1-i}(h_1,\ldots,h_{i-1},r_i,g_i,\ldots,g_n)g_{n+1} \big] d_{n+1}\\
=\Big[ (-1)^n\big[ (-1)^{n+1}(g_1,\ldots,g_n)\stackrel{\text{\color{blue}[5]}}{+}(-1)^n(r_1g_1,g_2,\ldots,g_n)\\
\stackrel{\text{\color{blue}[1]}}{+}\sum_{j=1}^{n-1}(-1)^{n-j}(r_1,g_1,\ldots,g_jg_{j+1},\ldots,g_n)
\stackrel{\text{\color{blue}[4]}}{+}(r_1,g_1,\ldots,g_{n-1})g_n \big]\\
+\sum_{i=2}^{n}(-1)^{n+1-i}\big[
\stackrel{\text{\color{blue}[3]}}{\phantom{+}}(-1)^{n+1}(h_2,\ldots,h_{i-1},r_i,g_i,\ldots,g_n)\\
\stackrel{\text{\color{blue}[2]}}{+}\sum_{j=1}^{i-2}(-1)^{n+1-j}(h_1,\ldots,h_jh_{j+1},\ldots,h_{i-1},r_i,g_i,\ldots,g_n)\\
\stackrel{\text{\color{blue}[5]}}{+}(-1)^{n-i}(h_1,\ldots,h_{i-2},h_{i-1}r_i,g_i,\ldots,g_n)
\stackrel{\text{\color{blue}[5]}}{+}(-1)^{n-i-1}(h_1,\ldots,h_{i-1},r_ig_i,\ldots,g_n)\\
\stackrel{\text{\color{blue}[1]}}{+}\sum_{j=i}^{n-1}(-1)^{n-j}(h_1,\ldots,h_{i-1},r_i,g_i,\ldots,g_jg_{j+1},\ldots,g_n)\\
\stackrel{\text{\color{blue}[4]}}{+}(h_1,\ldots,h_{i-1},r_i,g_i,\ldots,g_{n-1})g_n\big]\\
\stackrel{\text{\color{blue}[3]}}{+}(-1)^{n+1}(h_2,\ldots,h_n,r_{n+1})
\stackrel{\text{\color{blue}[2]}}{+}\sum_{j=1}^{n-1}(-1)^{n+1-j}(h_1,\ldots,h_jh_{j+1},\ldots,h_n,r_{n+1})\\
\stackrel{\text{\color{blue}[5]}}{-}(h_1,\ldots,h_{n-1},h_nr_{n+1})+(h_1,\ldots,h_n)r_{n+1}\Big]g_{n+1}.
\end{multline}

On the other hand, we have
\begin{multline}\label{gde}
(g_1,\ldots,g_n)g_{n+1}d_{n}\eta_{n-1}=\big[(-1)^n(g_2,\ldots,g_n)\\
+\sum_{i=1}^{n-1}(-1)^{n-i}(g_1,\ldots,g_ig_{i+1},\ldots,g_n)  +(g_1,\ldots,g_{n-1})g_n\big]g_{n+1}\eta_{n-1}
\end{multline}
We now make the following observation. The sets $\{h_1,\ldots,h_{n+1}\}$ and $\{r_1,\ldots,r_{n+1}\}$
were defined from $\{g_1,\ldots,g_{n+1}\}$ using (\ref{hrel}) and (\ref{rrel}). It can be readily seen that the corresponding sets defined from
$\{g_2,\ldots,g_{n+1}\}$  coincide with $\{h_2,\ldots,h_{n+1}\}$ and $\{r_2,\ldots,r_{n+1}\}$, and those defined from
$\{g_1,\ldots,g_ig_{i+1},\ldots,g_{n+1}\}$, $i=1,\ldots,n$, coincide with
$\{h_1,\ldots,h_ih_{i+1},\ldots,h_{n+1}\}$ and $\{r_1,\ldots,r_i,r_{i+2},\ldots,\allowbreak r_{n+1}\}$, respectively.
Therefore, we can use (\ref{eta_act}) to go on elaborating the right-hand side of (\ref{gde}) as follows:
\begin{multline}\label{gdec}
=
\stackrel{\text{\color{blue}[3]}}{\phantom{+}}(-1)^n\sum_{i=2}^{n+1}(-1)^{n+1-i}(h_2,\ldots,h_{i-1},r_i,g_i,\ldots,g_n)g_{n+1}\\
+\sum_{i=1}^{n-1}(-1)^{n-i}\big[
\stackrel{\text{\color{blue}[1]}}{\phantom{+}}\sum_{j=1}^{i}(-1)^{n-j}(h_1,\ldots,h_{j-1},r_j,g_j,\ldots,g_ig_{i+1},\ldots,g_n)g_{n+1}\\
\stackrel{\text{\color{blue}[2]}}{+}\sum_{j=i+1}^{n}(-1)^{n-j}(h_1,\ldots,h_ih_{i+1},\ldots,h_j,r_{j+1},g_{j+1},\ldots,g_n)g_{n+1}\big]\\
\stackrel{\text{\color{blue}[4]}}{+}\sum_{i=1}^{n-1}(h_1,\ldots,h_{i-1},r_i,g_i,\ldots,g_{n-1})g_ng_{n+1}
\stackrel{\text{\color{blue}[4]}}{+}(h_1,\ldots,h_{n-1},r_n)g_ng_{n+1}.
\end{multline}

Upon expanding the brackets we sum up the left- and right-hand sides of (\ref{getad}) and (\ref{gdec}).
Using relations (\ref{hrrg}) we see that the summands with the same blue marks cancel out. All that is left is
the following:
\begin{multline*}
(g_1,\ldots,g_n)g_{n+1}(\eta_n d_{n+1}+d_{n}\eta_{n-1})=-(g_1,\ldots,g_n)g_{n+1}+(h_1,\ldots,h_n)h_{n+1}\\
=(g_1,\ldots,g_n)g_{n+1}\psi_n.
\end{multline*}
This shows that $\eta$ is the required chain homotopy. \end{proof}

The homotopy relations (\ref{etap}) and (\ref{eta}) can be uniformly written as

\begin{equation}\label{etau}
  f'f - \operatorname{id}_S = \eta'd_+'+d'\eta_-' \quad \text{and} \quad
  ff' - \operatorname{id}_F = \eta d_+ +d\eta_-,
\end{equation}
where the subscript $+$ ($-$) means dimension shift by $+1$ ($-1$).

Now, given an arbitrary $H$-module $U$, we may apply the functor $\operatorname{Hom}_H(-,U)$, which
is additive and contravariant, to the truncated version of (\ref{FfS}) and to (\ref{etau}). This gives
\begin{equation}\label{HFfS}
\begin{array}{cccccc}
\ldots\stackrel{\partial_1}{\longleftarrow}&\operatorname{Hom}_H(F_1,U)                     &\stackrel{\partial_0}{\longleftarrow}&\operatorname{Hom}_H(F_0,U)&\longleftarrow& 0 \\[5pt]
                                           &\varphi_1\Big\uparrow\Big\downarrow \varphi_1'&                                      &\varphi_0\Big\uparrow\Big\downarrow \varphi_0' \\
\ldots\stackrel{\partial_1'}{\longleftarrow}&\operatorname{Hom}_H(S_1,U)                     &\stackrel{\partial_0'}{\longleftarrow}&\operatorname{Hom}_H(S_0,U)&\longleftarrow& 0
\end{array}
\end{equation}
and
\begin{equation}\label{Hetau}
  \varphi\varphi' - \operatorname{id}_{\operatorname{Hom}_H(S,U)} = \partial'\mu_+'+\mu'\partial_-' \quad \text{and} \quad
  \varphi'\varphi - \operatorname{id}_{\operatorname{Hom}_H(F,U)} = \partial\mu_++\mu\partial_-,
\end{equation}
where we denoted
$$
\partial^(\mbox{}'\mbox{}^)_-=\operatorname{Hom}_H(d^(\mbox{}'\mbox{}^),U), \quad
\varphi^(\mbox{}'\mbox{}^)=\operatorname{Hom}_H(f^(\mbox{}'\mbox{}^),U), \quad
\mu^(\mbox{}'\mbox{}^)_+=\operatorname{Hom}_H(\eta^(\mbox{}'\mbox{}^),U), \quad
$$
and ${}^(\mbox{}'\mbox{}^)$ means that $'$ is either present or omitted on both sides.
Relations (\ref{Hetau}) imply that the top and bottom cochain complexes in (\ref{HFfS}) are (co)homotopy equivalent.
Lemma \ref{heq} implies that their cohomology groups are isomorphic via the mutually inverse isomorphisms

\begin{equation}\label{PhiInv}
\xymatrix{
\operatorname{Ker}\partial/\operatorname{Im}\partial_- \ar@<-.5ex>@/^/[rr]|{\,\Phi'} &&
\operatorname{Ker}\partial'/\operatorname{Im}\partial_-' \ar@<-.5ex>@/^/[ll]|{\,\Phi\,}
}
\end{equation}
induced by $\varphi'$ and $\varphi$. In view of (\ref{fn}), we can write explicitly
\begin{equation}\label{phiexpl}
\varphi_n=[\beta \mapsto f_n\beta] =  [\beta \mapsto [(g_1,\ldots,g_n)g_{n+1}\mapsto \beta(h_1,\ldots,h_n)h_{n+1}]],
\end{equation}
$n\geqslant 0$, for all $\beta \in \operatorname{Hom}_H(S_n,U)$ and $g_i\in G$, $i=1,\ldots,n+1$;
and note that $\varphi'_n$ sends $\beta' \in \operatorname{Hom}_H(F_n,U)$ to its restriction
to $S_n\leqslant F_n$.

\section{Main proof}

In this section, we prove Theorem \ref{main}.

Let $G$ be a group, $H\leqslant G$ a subgroup, and let $U$ be a $\mathbb{Z}H$-module.
The co-induced module $W=\operatorname{Hom}_{H}(\mathbb{Z}G,U)$ is by definition the set
$$
\{\,f:G\to U\mid f(gh)=f(g)h, \ \ \forall\  g\in G, h\in H\,\}
$$
which becomes a right $G$-module with respect to the action
$$
(fa)(g)=f(ag), \quad \forall\  a,g\in G.
$$

Consider the standard free inhomogeneous resolution for $G$
\begin{equation}\label{xresg}
(X,l):\qquad \ldots \stackrel{l_2}{\longrightarrow}{X_1} \stackrel{l_1}{\longrightarrow}{X_0} \longrightarrow \mathbb{Z} \to 0
\end{equation}
whose restriction $(F,d)$ to $H$ is a free resolution for $H$, where $F=\operatorname{res}_H^G(X)=X\otimes_G \mathbb{Z}G$, $d=\operatorname{res}_H^G(l)$, and
we view $\mathbb{Z}G$ as a natural ($\mathbb{Z}G$,$\mathbb{Z}H$)-bimodule.
Applying to $F$ the functor $\operatorname{Hom}_H(-,U)$ gives the complex
\begin{equation}\label{hfu}
0\longrightarrow\operatorname{Hom}_H(F_0,U)\stackrel{\partial_{0}}{\longrightarrow}\operatorname{Hom}_H(F_1,U)
 \stackrel{\partial_{1}}{\longrightarrow}\ldots
\end{equation}
Since $F_i=\operatorname{res}_H^G(X)=X_i\otimes_G \mathbb{Z}G$ for $i\geqslant 0$, we may apply Lemma (\ref{thadj}) to have the
mutually inverse isomorphisms
\begin{equation}\label{psiinv}
\xymatrix{
\operatorname{Hom}_H(F,U) \ar@<-.5ex>@/^/[rr]|{\,\psi} &&
\operatorname{Hom}_G(X,W) \ar@<-.5ex>@/^/[ll]|{\,\psi'\,}.
}
\end{equation}
The explicit form (\ref{thexpl})
of these isomorphisms
can be rewritten in our case as
\begin{equation}\label{hgexpl}
\begin{aligned}
  \psi:\ \ \lambda&\ \mapsto\ [(g_1,\ldots,g_i)\mapsto[g_{i+1}\mapsto\lambda((g_1,\ldots,g_i)g_{i+1})]], \\
  \psi':\ \ \mu&\ \mapsto\  [(g_1,\ldots,g_i)g_{i+1}\mapsto \mu(g_1,\ldots,g_i)(g_{i+1})]
\end{aligned}
\end{equation}
for all $\lambda\in \operatorname{Hom}_H(F_i,U)$, $\mu\in \operatorname{Hom}_G(X_i,W)$, $g_1,\ldots, g_{i+1}\in G$.

\begin{lem}\label{pschm}
The homomorphism $\psi$ and $\psi'$ in $(\ref{psiinv})$ are (co)chain maps between the cochain
complexes $\operatorname{Hom}_H(F,U)$ and $\operatorname{Hom}_G(X,W)$.
\end{lem}
\begin{proof}
All we need to show is that the following diagram is commutative
\begin{equation}\label{HFpsiS}
\begin{array}{cccccc}
\ldots\stackrel{\partial_1}{\longleftarrow}&\operatorname{Hom}_H(F_1,U)                     &\stackrel{\partial_0}{\longleftarrow}&\operatorname{Hom}_H(F_0,U)&\longleftarrow& 0 \\[5pt]
                                           &\psi_1\Big\downarrow\Big\uparrow \psi_1'&                                      &\psi_0\Big\downarrow\Big\uparrow \psi_0' \\
\ldots\stackrel{\delta_1}{\longleftarrow}&\operatorname{Hom}_G(X_1,W)                     &\stackrel{\delta_0}{\longleftarrow}&\operatorname{Hom}_G(X_0,W)&\longleftarrow& 0
\end{array}
\end{equation}
where $\partial = \operatorname{Hom}_H(d,U)$ and $\delta = \operatorname{Hom}_G(l,W)$.
This follows from the fact that, for every $\alpha\in \operatorname{Hom}_H(F_i,U)$, we have the equality $\alpha\partial_i\psi_{i+1}=\alpha\psi_i\delta_i$
in $\operatorname{Hom}_G(X_{i+1},W)$. Indeed, $\alpha\partial_i\psi_{i+1}$ sends
\begin{multline*}
(g_1,\ldots,g_{i+1})\ \ \mapsto\ \ [g_{i+2}\mapsto(\alpha\partial_i)((g_1,\ldots,g_{i+1})g_{i+2})]\\
=[g_{i+2}\mapsto\alpha((g_1,\ldots,g_{i+1})g_{i+2}d_{i+1})]=\ [g_{i+2}\mapsto\alpha((g_1,\ldots,g_{i+1})d_{i+1}g_{i+2})],
\end{multline*}
while $\alpha\psi_i\delta_i$ sends
\begin{multline*}
(g_1,\ldots,g_{i+1})\ \ \mapsto\ \ (\alpha\psi_i)((g_1,\ldots,g_{i+1})l_{i+1})=\\
=\ [g_{i+2}\mapsto\alpha((g_1,\ldots,g_{i+1})l_{i+1}g_{i+2})]
\end{multline*}
for all $g_k\in G$. The results are the same, since $d=\operatorname{res}_H^G(l)$.
\end{proof}

Lemma \ref{pschm} implies that there are mutually inverse isomorphisms
\begin{equation}\label{PsiInv}
\xymatrix{
\operatorname{Ker}\delta/\operatorname{Im}\delta_- \ar@<-.5ex>@/^/[rr]|{\,\Psi\,} &&
\operatorname{Ker}\partial/\operatorname{Im}\partial_- \ar@<-.5ex>@/^/[ll]|{\,\Psi'\,}
}
\end{equation}
between the cohomology groups resulting from the complexes
$\operatorname{Hom}_H(F,U)$ and $\operatorname{Hom}_G(X,W)$
which are induced by the maps $\psi$ and $\psi'$.
We also know from the previous section that the cohomology groups of $\operatorname{Hom}_H(F,U)$
are isomorphic to those of $\operatorname{Hom}_H(S,U)$, where $S$ is the standard free resolution for $H$,
the isomorphism being induced by the chain maps $\varphi$ and $\varphi'$ given in
(\ref{phiexpl}). Combining (\ref{PhiInv}) and (\ref{PsiInv}) gives the Shapiro isomorphisms
\begin{equation}\label{ThetaInv}
\xymatrix{
\operatorname{Ker}\delta/\operatorname{Im}\delta_- \ar@<-.5ex>@/^/[rr]|{\,\Theta\,} &&
\operatorname{Ker}\partial'/\operatorname{Im}\partial'_- \ar@<-.5ex>@/^/[ll]|{\,\Theta^{-1}\,},
}
\end{equation}
where $\Theta=\Psi'\Phi'$ and $\Theta^{-1}=\Phi\Psi$. The latter is therefore induced by the composition $\phi\psi$ and, by (\ref{phiexpl}) and (\ref{hgexpl}), is expressed explicitly as
\begin{align*}
\beta\ &\stackrel{\phi_n}{\mapsto}\ [(g_1,\ldots,g_n)g_{n+1}\mapsto\beta(h_1,\ldots,h_n)h_{n+1}]\\
       &\stackrel{\psi_n}{\mapsto}\ [(g_1,\ldots,g_n)\mapsto [g_{n+1}\mapsto\beta(h_1,\ldots,h_n)h_{n+1}]]\in \operatorname{Ker}\delta_n=Z^n(G,W)
\end{align*}
for every $\beta\in \operatorname{Ker}\partial'_n=Z^n(H,U)$, as required. The proof of Theorem \ref{main} is complete.

\section{An application}\label{sec:appl}

In this section, we apply the previous results to solve constructively the following

\begin{prob}[V.\,D. Mazurov]\label{prVDM}
Does there exist an extension $E$ of a finite $2$-group $T$
by a Suzuki group $Sz(q)$ for some $q>2$ such that all involutions from $E$ lie in~$T$?
\end{prob}

To this end, we prove the next generalisation.

\begin{prop}\label{mprop2}
Let $G$ be a finite group with a unique conjugacy class of involutions and let $H\leqslant G$
be a subgroup of order $2$.  Let $U$ be the principal $\mathbb{F}_2H$-module and let $T=U^G$. Then there
is an extension $E$ of $T$ by $G$ such that all involutions of $E$ lie in $T$.
\end{prop}

\begin{proof}
Since $G$ is finite, we may identify $T$ with $\operatorname{coind}_H^G(U)$ which is the set
\begin{equation}\label{tdef}
\{\,f:G\to U\mid f(gh)=f(g)\ \forall g\in G,h \in H\,\}
\end{equation}
with a natural structure of an $\mathbb{F}_2$-space and the action of $a\in G$ given by
\begin{equation}\label{fag}
(fa)(g)=f(ag)
\end{equation}
for all $a,g\in G$. By Shapiro's lemma, there is an isomorphism
$$
  \Theta:\ H^n(G,T)\to H^n(H,U).
$$
Let $U=\{0,u\}$ and $H=\{1,h\}$. Define $\beta:H\times H\to U$ as follows:
\begin{equation}\label{bdef}
\beta(h_1,h_2)=\left\{
\begin{array}{rl}
  u, & \text{if}\ \ h_1=h_2=h;\\
  0, & \text{otherwise}.
\end{array}
\right.
\end{equation}
It is readily checked that
$$
\beta(h_2,h_3)-\beta(h_1h_2,h_3)+\beta(h_1,h_2h_3)-\beta(h_1,h_2)=0
$$
for all $h_1,h_2,h_3\in H$; i.\,e., $\beta\in Z^2(H,U)$. Observe that $\beta\not\in B^2(H,U)$, which can be verified either by definition or
by noting that $\beta$ defines the nonsplit extension
$$
0\to U \to \mathbb{Z}_4 \to H \to 1.
$$
In particular $\beta+B^2(H,U)$ is a nonzero element of $H^2(H,U)$.

Let $Y$ be a left transversal of $H$ in $G$ and let $\pi:G\to H$
be the projection along~$Y$; i.\,e., $\pi(yh)=h$ for all $y\in Y$, $h\in H$. In particular,
\begin{equation}\label{pgh}
  \pi(gh)=\pi(g)h
\end{equation}
for all $g\in G,\ h\in H$. We may assume that
\begin{equation}\label{pi1}
  \pi(1)=1.
\end{equation}
Let $\alpha\in Z^2(G,T)$ be the image of $\beta$ under the map defined in Theorem~\ref{main} which induces
the inverse Shapiro isomorphism $\Theta^{-1}$. Explicitly, we have
\begin{equation}\label{ag1g2}
\alpha(g_1,g_2) = [g\mapsto \beta(\pi(g_1g_2g)\pi(g_2g)^{-1},\pi(g_2g)\pi(g)^{-1})]
\end{equation}
for all $g_1,g_2,g\in G$. Observe that $\alpha+B^2(G,T)$ is a
nonzero element of $H^2(G,T)$ and so $\alpha$ defines a certain nonsplit extension
$$
0\to T \to E\to G \to 1.
$$
This extension can be defined as the set
$$
E=\{(g,f)\mid g\in G, f\in T\}
$$
with the multiplication
$$
(g_1,f_1)(g_2,f_2)=(g_1g_2,f_1g_2+f_2+\alpha(g_1,g_2))
$$
for all $g_1,g_2\in G$, $f_1,f_2\in T$.
This is indeed a group, because $\alpha$ is a normalised $2$-cocycle. We will identify $T$ with the
set of pairs $\{(1,f)\mid f\in T\}\subseteq E$.

It remains to see that every involution of $E$ lies in $T$. Since all involutions of $G$ are conjugate (to $h\in H$),
it suffices to show that $(h,f)^2\ne 1$ for all $f\in T$. We have $(h,f)^2=(1,l)$, where $l=fh+f+\alpha(h,h)\in T$.
We show that $l(1)\ne 0$, which will yield $l\ne 0$ as required. We have
\begin{multline}
l(1)=(fh)(1)+f(1)+\alpha(h,h)(1)\stackrel{(\ref{fag}),(\ref{ag1g2})}{=}
f(h)+f(1)+\beta(\pi(1)\pi(h)^{-1},\pi(h)\pi(1)^{-1})\\
\stackrel{(\ref{tdef}),(\ref{pgh}),(\ref{pi1})}{=}f(1)+f(1)+\beta(h^{-1},h)\stackrel{(\ref{bdef})}{=}u\ne 0.
\end{multline}
The proof is complete. \end{proof}

Proposition \ref{mprop2} clearly implies Proposition \ref{mprop} and also answers affirmatively
the question in Problem \ref{prVDM}, because the Suzuki groups $Sz(q)$ have a unique conjugacy
class of involutions.

{\em Acknowledgment.}
This research was carried out within the State Contract of the Sobolev Institute of Mathematics (project FWNF-2022-0002).

\end{document}